\documentclass[11pt,reqno]{amsart}

\usepackage[T1]{fontenc}
\usepackage{graphicx}

\usepackage{color}
\definecolor{MyLinkColor}{rgb}{0,0,0.4}

\newcommand{\R}{{\mathbb R}}

\newcommand{\N}{{\mathbb N}}
\newcommand{\s}{\mathbb S}
\newcommand{\W}{\mathcal{W}}
\newcommand{\U}{\mathcal{U}}

\newcommand{\cS}{\mathcal{S}}
\newcommand{\cC}{\mathcal{C}}
\newcommand{\re}{\mathop{\rm Re}\nolimits}
\newcommand{\sign}{\mathop{\rm sign}\nolimits}
\newcommand{\ov}{\overline}
\newcommand{\p}{\partial}
\newcommand{\0}{\Omega}
\newcommand{\G}{\Gamma}
\newcommand{\e}{\varepsilon}
\newcommand{\tet}{\theta_{\gamma,\alpha}}

\newtheorem{thm}{Theorem}[section]
\newtheorem{prop}[thm]{Proposition}
\newtheorem{lemma}[thm]{Lemma}
\newtheorem{cor}[thm]{Corollary}

\theoremstyle{remark} 
\newtheorem{rem}[thm]{Remark}

\numberwithin{equation}{section}
\title[Steady-state fingering patterns]{Steady-state fingering patterns for a periodic Muskat problem}


\author[M. Ehrnstr\"om]{Mats Ehrnstr\"om}

\address{Department of Mathematical Sciences, Norwegian University of Science and Technology, 7491 Trondheim, Norway}
\email{mats.ehrnstrom@math.ntnu.no}

\author[J. Escher]{Joachim Escher}
\author[B.--V. Matioc]{Bogdan--Vasile Matioc}

\address{Institut f{\"u}r Angewandte Mathematik, Leibniz Universit{\"a}t Hannover, Welfengarten~1, 30167 Hannover, Germany. }
\email{escher@ifam.uni-hannover.de}
\email{matioc@ifam.uni-hannover.de}

\usepackage[colorlinks=true,linkcolor=MyLinkColor,citecolor=MyLinkColor]{hyperref} 

\begin{document}

\begin{abstract}
We study global bifurcation branches consisting of stationary solutions of the 
Muskat problem. It is proved that the steady-state fingering patterns blow up as the surface tension increases: we find a threshold value for the cell height with the property that 
below this value the fingers will touch the boundaries of the cell when the surface tension approaches a finite value from below; otherwise, the maximal slope of the fingers tends to infinity.
\end{abstract}

\maketitle

\section{Introduction}\label{sec:intro}
Proposed in $1934$ by Muskat (cf. \cite{M}), the Muskat problem describes the evolution of the interface
 between to immiscible fluids in a porous medium. 
In the recent investigation \cite{EM12} this problem was studied in a new, periodic, setting incorporating gravity, viscosity, and surface tension effects. 
The current note aims at an in-depth description of the stationary solutions found in that work. 

When a heavier viscous fluid rests upon a lighter one, the interface between them is in general not stable; depending on the different densities, and the surface tension, one expects the upper fluid to, at least partially, sink into the lower one, and vice versa. 
Due to their resemblance to an outstretched hand reaching into a viscous fluid, the resulting shapes are often referred to as \emph{fingering patterns}. 
The investigation of such, in different settings, has brought a 
lot of attention (see, e.g., the pioneering paper \cite{MR0097227} 
and the later investigations \cite{MR728709,MS1981,MR1462048,MR877015}.

In \cite{ EM12} smooth branches of stationary, i.e. time-independent, solutions of the Muskat problem were found.
They  are periodic solutions of the Laplace-Young equation under a volume constrain (see \eqref{eq:mean}).
The Laplace-Young equation is also known as the capillarity equation and, subjected to boundary constrains,  has been studied
by many authors (see \cite{RF} and the literature therein).
 
The  solutions we found are all even, but only in a small neighbourhood of the trivial solutions can one via linearisation 
obtain an approximate picture of the fingering patterns. This is due to the fact that global bifurcation theorems 
are inherently implicit in nature, and thus have the drawback of not disclosing the behaviour of the bifurcation 
branches away from the bifurcation point. In our present work, we therefore take advantage of the theory for ordinary 
differential equations and certain symmetry properties of the solutions to give a precise description of the solutions 
found in \cite{EM12}: we show that each global bifurcation branch consists entirely of 
steady-state solutions of minimal period $2\pi/l$, $l \in \N$, and that the symmetric fingers described by the interface i) either approach the 
bottom and the upper boundary of the cell, or ii) display blow-up in the $C^1-$norm, while 
the surface tension coefficient tends from below to a finite value.   

The plan is as follows. In Section~\ref{sec:prel} we give the
 necessary mathematical background of the problem, and show that,
 for stationary solutions, it may be reduced to an ordinary differential equation 
with an additional non-local constraint. 
The proof of the main result mentioned above 
is based on the study of the odd solutions of this equation, and the  one-to-one correspondence between the odd and the even solutions thereof. 
This is done in the Section~\ref{sec:odd}, and there we also show that there exist infinitely 
many global bifurcation branches consisting of odd solutions of the problem. In addition, we describe the behaviour 
of the steady fingers away from the set of trivial solutions.
Finally, it is interesting to see that  the steady-state fingering patterns we obtained correspond to certain 
solutions of the mathematical pendulum.  
This correspondence is shown in the Appendix.

\begin{figure}
\includegraphics[clip=true, angle=0,  trim = 0 -10 0 0, width=0.4\linewidth]{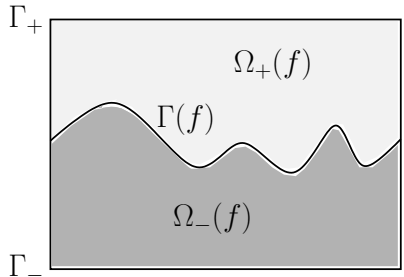}
\caption{\small The periodic and vertical Hele-Shaw cell.}
\label{Figure1}
\end{figure}

\section{Preliminaries}\label{sec:prel}
Let $h > 0$, and consider a periodic medium occupying $\s \times [-h,h]$, with $\s$ denoting the unit circle.
 The bottom  of this cell is assumed to be impermeable, 
and the pressure on the upper boundary is constantly set to zero. 
For a function $f$ with $\|f\|_{C(\s)}<h$, let
\[
\Gamma(f(t)):=\{(x,f(t,x))\,:\,x\in\s\},
\]
be the time-dependent interface separating the wetting phases, and $\Gamma_{\pm}:=\s\times\{\pm h\} $ the bottom and
 the upper boundary of the cell (see Figure~\ref{Figure1}). We define the fluid domains
\begin{align*}
\0_-(f(t))&:=\{(x,y)\,:\, -h<y<f(t,x)\},\\[1ex]
\0_+(f(t))&:=\{(x,y)\,:\, f(t,x)<y<h\},
\end{align*}
and write 
\[
\kappa_{\Gamma(f)} := \frac{f_{xx}}{\left(1+{f_x^2}\right)^{3/2}} 
\]
for the signed curvature of the graph $\Gamma(f)$. The mathematical model can then be stated as a two-phase moving-boundary problem,
\begin{equation}\label{eq:S}
\left\{
\begin{array}{rlllll}
\Delta u_\pm&=&0&\text{in}& \Omega_\pm(f(t)), \\[2ex]
\p_\nu u_+&=&g_1&\text{on}& \Gamma_+,\\[2ex]
 u_-&=&g_2&\text{on}&\Gamma_{-},\\[2ex]
u_+-u_-&=&\gamma\kappa_{\G(f)}+g(\varrho_+-\varrho_-)  f&\text{on}& \Gamma(f(t)),\\[2ex]
\p_tf&=&-\displaystyle{\frac{\sqrt{1+f_x^2}}{\mu_\pm}\p_\nu u_\pm}&\text{on}& \Gamma(f(t)),\\[2ex]
f(0)&=&f_0, &
\end{array}
\right. 
\end{equation} 
with $t\in[0,T],$ where we use the subscripts $\pm$ to denote the upper and lower fluids, respectively. As  conventional, $g$ stands for the gravitational constant of acceleration, $\gamma$ denotes the surface tension at the interface $\Gamma(f)$, and $\varrho_\pm$ and $\mu_{\pm}$ are the densities and viscosities of the two fluids, respectively, all of which are supposed to be given positive constants. 
Physically, the potentials $u_\pm$ are defined by the relation 
\[
u_\pm:=p_\pm+g\varrho_\pm y,
\] 
where $p$ stands for pressure, and $y$ is the height coordinate.   
Furthermore, the functions $g_1$ and $g_2$ are assumed to be known 
\[
\text{$g_1\in C([0,T], h^{1+\alpha}(\mathbb{S})) $ and $g_2\in C([0,T], h^{2+\alpha}(\mathbb{S})) $.}
\]
Given $m\in\N$ and $\alpha\in(0,1),$ the small H\"older spaces
 $h^{m+\alpha}(\s)$ stand for the completions of the class of smooth functions in the Banach spaces $C^{m+\alpha}(\s).$

The problem consists of finding functions $f$ and $u_\pm$ satisfying \eqref{eq:S}, but 
it can be shown that this may be reduced to a parabolic problem with   $f$ as the single unknown  \cite{EM12}. 
Hence, we shall refer to the function $f$ parametrising the moving interface between the fluids as a solution of \eqref{eq:S}.

\subsection*{Well-posedness results}
It is shown in \cite{EM12} that the Muskat problem  is, at least in a neighbourhood of some flat interface, of parabolic type.
This observation is true, when considering surface tension effects, independently of the boundary data $g_1$ and $g_2$.
On the other hand, when neglecting surface tension, certain restrictions  must be imposed on the  boundary data to ensure parabolicity  of \eqref{eq:S}.
We then have (cf. \cite[Theorem 2.1]{EM12}):

\begin{thm}[Well-posedness] Let $\gamma\in[0,\infty)$, $c_1, c_2\in\mathbb{R}$, and assume that
\begin{equation}\label{eq:cond}
\gamma>0\qquad\text{or}\qquad g(\rho_+-\rho_-)+c_1\left(\frac{\mu_-}{\mu_+}-1\right)<0.
\end{equation}
Then there exist open  neighbourhoods of the zero function $\mathcal{O}_i\subset h^{i+\alpha}(\mathbb{S})$, $i\in\{1,2\},$ and $\mathcal{O}\subset h^{2+2\sign(\gamma)+\alpha}(\mathbb{S})$,
 such that for all 
 $f_0\in \mathcal{O}$ and $g_i\in C([0,\infty), c_i+\mathcal{O}_i),$  $i=1,2,$
there exists $T(f_0)\leq T$ and a unique maximal H\"older   solution $f$ of problem \eqref{eq:S} on  $[0,T(f_0))$ which fulfills $f(t)\in\mathcal{O}$ for all $t\in[0,T(f_0)).$

If  $\gamma>0$, then we may choose $\mathcal{O}_i=h^{i+\alpha}(\mathbb{S})$, $i\in\{1,2\}.$
\end{thm}

Existence of classical solutions of the Muskat problem, and long-time existence for small initial data, can also be found 
 in \cite{FT, SCH, Y1, Yi2}. 
 The approach in \cite{EM12} yields structural insight 
into the character of the Muskat and it is  suitable for studying the  stability properties of the steady-state solutions  of problem \eqref{eq:S}.

\subsection*{Steady-state solutions}
In the remainder of this paper we assume that $g_1\equiv0 $ and $g_2\equiv {\rm const},$ meaning that the mass of both fluids is preserved in time,
and that the cell contains equal quantities of both fluids.
The steady-state solutions of \eqref{eq:S} are then solutions of the problem
\begin{equation}\label{eq:mean}
\gamma\frac{f''}{(1+f'^2)^{3/2}}+g(\varrho_+-\varrho_-)f=const, \quad\text{and}\quad \int_\s f\, dx=0.
\end{equation}
Indeed, since $f$ does not depend on time, it follows from uniqueness for the 
Dirichlet--Neumann problem that the potentials $u_+$ and $u_-$ are both constants also in the 
spatial variable, which yields the first equation of \eqref{eq:mean}. The second relation reflects the 
earlier mentioned assumption that the cell contains equal amounts of both fluids. By induction, we obtain 
\begin{rem}\label{R:1}
Any classical solution of \eqref{eq:mean} is smooth.
\end{rem}
We shall refer to the set 
\[
\Sigma:=\{(\gamma,0)\,:\, \gamma>0\} 
\]
as  being the trivial branch of solutions of \eqref{eq:mean}. Because of the integral constraint in \eqref{eq:mean}, the problem \eqref{eq:mean} is in general over-determined. One way to approach this difficulty is to determine solution pairs $(\gamma, f)\in(0,\infty)\times C^2(\s)$ of \eqref{eq:mean}, under the additional, but natural, requirement that $\|f\|_{C(\s)}<h$, meaning that the fingers do not touch the lower or upper boundaries of the cell. 

In the situation when the less dense fluid lies on the bottom of the cell, i.e. when $\varrho_+ > \varrho_-$,
we find---using the theorem on bifurcation from simple eigenvalues due to Crandall and  Rabinowitz \cite[Theorem 1.7]{CR}, and the global bifurcation theorem due to Rabinowitz \cite[Theorem II.3.3]{HK}---global bifurcation branches consisting of even, stationary, finger-shaped solutions of \eqref{eq:mean}.
More precisely, if $C^{3+\alpha}_{0,e}(\s)$ denotes the subspace of  $C^{3+\alpha}(\s)$ 
consisting of even functions with integral mean zero, and 
\[
\W:=\left\{ f\in C^{3+\alpha}_{0,e}(\s)\,:\, \|f\|_{C(\s)}<h \right\},
\]
we have (cf. \cite[Theorem 6.1 and Theorem 6.3]{EM12}):

\begin{thm}[Bifurcation of stationary solutions]\label{T:P3} Let $g_1\equiv0 $, $g_2\equiv {\rm const},$  and $\varrho_+>\varrho_-$, $1 \leq l \in \N$.
The point 
\begin{equation}\label{eq:G}
(\ov\gamma_l,0):=(g(\varrho_+-\varrho_-)/l^2,0)
\end{equation}
belongs to the closure $\cS$ of the set of nontrivial solutions of \eqref{eq:mean} in $(0,\infty)\times\W.$
Denote by $\cC_l$ the connected component of $\cS$ to which $(\ov\gamma_l,0)$ belongs.
Then  $\cC_l$ has, in a  neighbourhood of $(\ov\gamma_l,0),$  an analytic parametrisation
$(\gamma_l,f_l):(-\delta,\delta)\to (0,\infty)\times\W$, 
\begin{align*} 
\gamma_l(\e)&=\ov\gamma_l +\frac{3 g(\varrho_+-\varrho_-)}{8}\e^2+O(\e^3),\\[1ex]
f_l(\e)&=\e\cos(lx)+O(\e^2),
\end{align*}
as $\e\to0.$ 
Any other pair $(\gamma,0)$, $\gamma > 0$, belongs to a neighbourhood in $(0,\infty) \times \W$ with only trivial solutions of \eqref{eq:mean}.

Furthermore, if $\e$ is small and $\gamma=\gamma_l(\e),$ then $f_l(\e)$ is an unstable stationary solutions of \eqref{eq:S}.
\end{thm}

Theorem~\ref{T:P3} is obtained by differentiating the first relation of \eqref{eq:mean} 
and finding in this way an equation for $f$ only (this is why solutions in $C^{3+\alpha}_{0,e}(\s)$ are considered). 
It is not difficult to show  that if $\varrho_-\geq \varrho_+,$ then  \eqref{eq:mean}
 has only the trivial solution $f=0$ (see, e.g., \cite{EM5}). In the paper at hand we show (cf.~Remark~\ref{R:2}) 
that this is the case when $\varrho_-< \varrho_+$ too, as long as the surface tension coefficient is large enough.

\section{Odd steady-state fingering solutions}\label{sec:odd}
In this section we consider the odd solutions of \eqref{eq:mean}. 
If $f$ is an odd function  on $\s$, then $f$ has integral mean $0$ and  $f(0)=f^{\prime\prime}(0)=0$.
Hence, the odd steady states of the Muskat problem~\eqref{eq:S} are exactly the odd solutions of the equation
\begin{equation}\label{eq:m}
\frac{f''}{(1+f'^2)^{3/2}}+\lambda f=0, \qquad \lambda > 0,
\end{equation}
within the set 
\[
\U:=\left\{ f\in C^{\infty}(\s) \colon \|f\|_{C(\s)}<h \right\}. 
\]
Here, the shorthand
\begin{equation}\label{eq:lambda}
\lambda := \frac{g(\varrho_+ - \varrho_-)}{\gamma},
\end{equation}
indicates the character of \eqref{eq:m} as an eigenvalue problem. 
\emph{Notice that, throughout this work, we consider the 
unstable case when $\varrho_+ > \varrho_-$, i.e. when the heavier fluid occupies the upper part of the membrane.}
 Equation \eqref{eq:m} admits the following scaling property: 

\begin{prop}\label{P:1} If $(\gamma,f) $ defines a solution of~\eqref{eq:m} through~\eqref{eq:lambda}, then 
\begin{equation}\label{eq:scaling}
\left( l^{-2}\gamma, l^{-1}f(l\cdot) \right),\qquad l\in\N,
\end{equation}
is also a solution of \eqref{eq:m}.
\end{prop} 
\begin{proof}
Since $\lambda$ is inversely proportional to $\gamma$, the result is immediate.
\end{proof}

The main result of this work is the following theorem, which states that a global bifurcation branch consisting of 
odd functions of minimal period $2\pi $ emanates from the trivial branch of solutions  $\Sigma$ at $(\ov\gamma_1,0),$
where $\ov\gamma_1$ is defined by \eqref{eq:G}.
This will later be used to characterise the global bifurcation branches of odd solutions which arise at $(\ov\gamma_l,0), l\geq 2,$ (see Corollary~\ref{C:1}
below),
and in Section~4 to describe the global bifurcation branches $\cC_l$ obtained in Theorem~\ref{T:P3}. 

Recall the definition of the beta function,
\[
B(x,y) := \int_0^1 t^{x-1}(1-t)^{y-1}\, dt,\qquad \re x, \re y>0.
\]

\begin{thm}\label{T:main}
For each $h>0$, there exists 
\begin{equation}\label{eq:g*}
\lambda_h \geq  \lambda_* := \frac{1}{2 \pi^2} \left(B\left(\textstyle{\frac{3}{4},\frac{1}{2}}\right)\right)^2,
\end{equation}
and corresponding $\gamma_h \leq\gamma_*$ defined by~\eqref{eq:lambda}, with 
the property that the nontrivial odd solutions of~\eqref{eq:m}  of minimal period  $2\pi$ within $\U$
coincide with the global bifurcation  curve  
\[
\Sigma_1:=\{(\gamma, \pm f_\gamma)\,:\, \gamma\in (\ov \gamma_1,\gamma_h)\},
\]
where the odd function  $f_\gamma\in C^\infty(\s)$ is uniquely determined by the parameter $\gamma\in (\ov \gamma_1,\gamma_*)$
 if we require that $f_\gamma'(0)\geq0.$
 Let $h_* := \sqrt{2/\lambda_*}$, and let  $f_\gamma$ denote 
the solution of \eqref{eq:m} of minimal period $2\pi$ (not necessarily in $\U$). 
The mapping $(\ov \gamma_1,\gamma_*)\times\s\ni (\gamma,x)\mapsto f_\gamma(x)$ is smooth, and 
\begin{itemize}
\item[$(i)$] if $h < h_*$, then $\gamma_h<\gamma_*$, and 
\[
\|f_\gamma\|_{C(\s)}=f_\gamma(\pi/2)\nearrow h \quad\text{ as }\quad \gamma\nearrow \gamma_h;
\] 

\item[$(ii)$] if $h = h_*$, then $\gamma_h=\gamma_*$, and 
\[
\|f_\gamma\|_{C(\s)}=f_\gamma(\pi/2)\nearrow h, \qquad \|f_\gamma^\prime\|_{C(\s)}=f'_\gamma(0)\nearrow \infty \quad\text{ as }\quad \gamma\nearrow \gamma_h;
\]

\item[$(iii)$] if $h > h_*$, then $\gamma_h=\gamma_*$, and 
\[
\|f_\gamma^\prime\|_{C(\s)} = f'_\gamma(0)\nearrow\infty \quad\text{ as }\quad \gamma\nearrow \gamma_h,
\]
while $\sup_{[\ov \gamma_1,\gamma_*)}\|f_\gamma\|_{C(\s)}<h$.
\end{itemize}
\end{thm}

Recall that $\ov\gamma_l,$ $1\leq l\in\N,$ is the constant defined by \eqref{eq:G}.
Combining Proposition~\ref{P:1}  and Theorem~\ref{T:main} we conclude:

\begin{cor}\label{C:1} 
Let $2 \leq l\in\N$. There exists $\gamma_{l,h}\in(\ov \gamma_1,\gamma_*]$ with the property that
\[
\Sigma_l:=\left\{ (l^{-2}\gamma, \pm l^{-1}f_\gamma(l\cdot))\,:\, \gamma\in (\ov \gamma_1,\gamma_{l,h}) \right\}
\]
consists exactly of the nontrivial odd solutions of minimal period $2\pi/l$ of \eqref{eq:m} within $\U.$
The alternatives $(i)-(iii) $ of Theorem~\ref{T:main} hold true with the natural modifications.
The disjoint union
\begin{equation}\label{union}
\mathcal{S}_{2\pi}:=\left(\cup_{l=1}^\infty\Sigma_l\right)\cup\left(\cup_{l=1}^\infty\left[\left(\ov \gamma_{l+1},\ov\gamma_l\right)\times\{0\}\right]\right)
\cup\left[(\ov\gamma_1,\infty)\times\{0\}\right]
\end{equation}
constitutes all nontrivial $2\pi-$periodic and odd solutions of \eqref{eq:m} in $\U$.
\end{cor}

\begin{figure}
\includegraphics[clip=true, angle=0,  trim = 0 -10 0 0, width=0.4\linewidth]{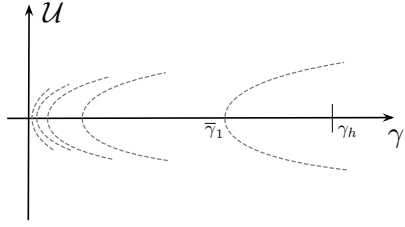}
\caption{\small A qualitative picture of the $2\pi$-periodic solutions described in Corollary~\ref{C:1}.}
\label{Figure2}
\end{figure}

\begin{rem}
Put differently, Corollary \ref{C:1} states that global bifurcation branches consisting of odd solutions emanate from $\Sigma$
 at $\ov\gamma_l,$ $1 \leq l\in\N$. Moreover, these bifurcation branches are pairwise disjoint. 
\end{rem}

\begin{rem}\label{rem:bifurcation}
It is worth mentioning that, for the same $\gamma$,
 we may find $2\pi$-periodic odd solutions of \eqref{eq:m} of 
different minimal periods (see Figure~\ref{Figure2}). 
Since $\gamma_{l,h}=\gamma_*>\ov\gamma_1$, for $l$  large enough, 
 there exists positive integers $l\in\N,$ such that
\[
\ov\gamma_{l+1}<\ov\gamma_l<\frac{\gamma_*}{(l+1)^2}<\frac{\gamma_*}{l^2}.
\]
Consequently equation \eqref{eq:m} 
possesses a solution which belongs to $\Sigma_l$ and another one in  $\Sigma_{l+1},$ corresponding to the same $\gamma$. 
\end{rem}

In order to prove Theorem~\ref{T:main} we need some preliminary results.

\begin{prop}\label{P:eau}
Let $\lambda > 0$ and $\alpha \in \R$ be given. The initial-value problem 
\begin{equation}\label{eq:ivp}
\left\{
\begin{array}{rllll}
\displaystyle \frac{f''}{(1+f'^2)^{3/2}}+\lambda f&=&0&\text{on}&\s,\\[1ex]
f(0)&=&0,\\[1ex]
f'(0)&=&\alpha
\end{array}
\right.
\end{equation}
possesses a unique classical solution $f_{\lambda,\alpha}$.
The solution is odd and periodic in $x$, and smooth as a map 
\[
(0,\infty)\times \R \times\R\mapsto f_{\lambda,\alpha}(x).
\]
\end{prop}
\begin{proof}
Setting $g:=f'$, we  rewrite \eqref{eq:ivp} as an initial value problem for the pair $(f,g),$ 
\begin{equation}\label{eq:ivp2}
\left(
\begin{array}{c}
f\\
g
\end{array}
\right)'
=F
\left(
\begin{array}{c}
f\\
g
\end{array}
\right),\quad
\left(
\begin{array}{c}
f\\
g
\end{array}
\right)(0)=
\left(
\begin{array}{c}
0\\
\alpha
\end{array}
\right),
\end{equation}
where $F:\R^2\to\R^2 $ is defined by
\[
F
\left(
\begin{array}{c}
f\\
g
\end{array}
\right)=
\left(
\begin{array}{c}
g\\
-\lambda f(1+g^2)^{3/2}
\end{array}
\right).
\]
Since $F$ is smooth,
there is a unique and smooth solution of \eqref{eq:ivp2}, defined on a maximal interval $[0,T)$; if $T<\infty$, then the solution blows up, meaning that $\sup_{[0,T)}|(f,g)|=\infty$ (cf. \cite{A}).

Notice that if $f$ is an odd solution of \eqref{eq:m} with slope $f'(0)=\alpha>0$,
then $-f$ is also an odd solution  of~\eqref{eq:m} with slope $-\alpha$.
Without loss of generality we may therefore restrict our attention to solutions of \eqref{eq:m} with nonnegative slope at $x=0$. Clearly,  the solution of \eqref{eq:ivp} with slope $\alpha=0$ is $f \equiv 0$. 

Suppose now that $\alpha>0$. We prove that there exists a  positive constant $\theta_{\lambda,\alpha}$ such that $f'>0$ on $[0,\theta_{\lambda,\alpha})$ and $f'(\theta_{\lambda,\alpha})=0.$
Indeed, assuming the contrary, we obtain in view of $f'(0)=\alpha>0$, that $f'>0$ on $[0,T)$.

On the one hand, if $T=\infty$,  we infer from \eqref{eq:m} that
\[
0=f^{\prime\prime}(x) +\lambda f(x)  \left( 1+(f^\prime(x))^2\right)^{3/2}\geq f^{\prime\prime}(x)+\lambda f(1), \quad\text{for all }\quad x\geq 1.
\]
Integration yields that
\[
f'(x)\leq f'(1)-\lambda f(1)(x-1)\to -\infty \quad\text{ as }\quad x\to\infty,
\]
which contradicts our assumption.

On the other hand, if $T<\infty,$ then either $\sup_{[0,T)}f=\infty$ or $\sup_{[0,T)}f'=\infty$, 
the latter case being excluded by the fact that $f'$ is decreasing for positive $f$. If $\sup_{[0,T)}f=\infty$, we multiply \eqref{eq:m} by $-f'$ and integrate over $[0,x]$ to obtain that
\begin{equation}\label{eq:u1}
\frac{1}{(1+f'^2(x))^{1/2}}=\frac{1}{(1+\alpha^2)^{1/2}}+\frac{\lambda f^2(x)}{2}, \qquad 0 < x<T.
\end{equation} 
Letting $x\to T$, we obtain the desired contradiction. Consequently, there exists a unique $\theta_{\lambda,\alpha}>0$, such that $f'>0$ on $[0,\theta_{\lambda,\alpha})$ and $f'(\theta_{\lambda,\alpha})=0.$

It can be easily seen that  $f$ extends to an odd function of minimal period $T_{\lambda,\alpha}:=4 \theta_{\lambda,\alpha}$. Indeed, we see that
\begin{equation}\label{eq:fdef}
f_{\lambda,\alpha}(x):=\left\{
\begin{array}{rllll}
f(x),&&  0\leq x\leq \theta_{\lambda,\alpha},\\[1ex]
f(2\theta_{\lambda,\alpha}-x),&& \theta_{\lambda,\alpha} \leq x\leq 2\theta_{\lambda,\alpha},\\[1ex]
-f(x-2\theta_{\lambda,\alpha}),&& 2\theta_{\lambda,\alpha} \leq x\leq 3\theta_{\lambda,\alpha},\\[1ex]
-f(4\theta_{\lambda,\alpha}-x),&& 3 \theta_{\lambda,\alpha}\leq x\leq 4 \theta_{\lambda,\alpha},
\end{array}
\right.
\end{equation}
has an odd and $T_{\lambda,\alpha}-$periodic extension on the whole of $\R$.
\end{proof}

We now  explicitly determine the minimal period, called $T_{\lambda,\alpha}$, of the solution $f_{\lambda,\alpha}$ of \eqref{eq:ivp}. In order to simplify calculations, we put
\[
\beta := \frac{1}{\sqrt{1+\alpha^2}}.
\]
From relation \eqref{eq:u1}, we find for $x=\theta_{\lambda,\alpha}$ that the maximum of $f_{\lambda,\alpha}$ is
\begin{equation}\label{eq:u2}
f_{\lambda,\alpha}(\theta_{\lambda,\alpha})=\sqrt{2 \lambda^{-1} \left(1-\beta \right)}.
\end{equation}
We also infer from the same relation  that 
\[
f_{\lambda,\alpha}'(x)=\sqrt{\left(\beta+\frac{\lambda f^2_{\lambda,\alpha}(x)}{2}\right)^{-2}-1}\quad
\text{ for all $x\in[0,\tet]$.}
\]
Dividing this equality by its right-hand side, we find that
\[
\theta_{\lambda,\alpha} = \int_0^{\theta_{\lambda,\alpha}} f_{\lambda,\alpha}^\prime(x)
\left(\left(\beta+\frac{\lambda f^2_{\lambda,\alpha}(x)}{2}\right)^{-2}-1\right)^{-1/2}\, dx,
\] 
and the variable substitution $f(x)=s$ yields 
\[
\theta_{\lambda,\alpha}=\int_0^{f_{\lambda,\alpha}(\theta_{\lambda,\alpha})}
\left(\left(\beta + \frac{\lambda s^2}{2}\right)^{-2}-1\right)^{-1/2}\, ds.
\]
Finally, setting $\tau:=s /f_{\lambda,\alpha}(\theta_{\lambda,\alpha})$, we obtain in virtue of \eqref{eq:u2} that
\begin{equation}\label{eq:tet}
\theta_{\lambda,\alpha}=\sqrt{\frac{2}{\lambda}}\int_0^1\frac{(1-\beta) \tau^2 + \beta}{\sqrt{(1-\tau^2)
\left[1+ (1-\beta) \tau^2+ \beta  \right]}}\, d\tau,
\end{equation}
for all $\alpha,\lambda>0.$
Since $\alpha \mapsto \beta$ is smooth, we may extend  $\theta_{\lambda,\alpha}$ continuously to   $(0,\infty)\times[0,\infty)$.  
More precisely, we state:

\begin{lemma}\label{L:mon}
The function $\theta_{\lambda,\alpha}$ defined in \eqref{eq:tet},
\[
(0,\infty)\times[0,\infty)\ni (\lambda,\alpha)\mapsto \theta_{\lambda,\alpha} \in(0,\infty),
\]
is smooth, and strictly decreasing with respect to both $\lambda$ and $\alpha$. Moreover\footnote{Let $\gamma>0$ be fixed. Recall that, given $\alpha >0$, the value $T_{\lambda,\alpha}=4 \theta_{\lambda,\alpha}$ denotes the minimal period of the solution $f_{\lambda,\alpha}$ of \eqref{eq:ivp} and that the latter problem possesses the trivial solution $f_{\lambda,0}\equiv 0$ if $\alpha=0$. Having said this, it is clear that the value $\theta_{\lambda,0}=\pi/(2\sqrt{\lambda})$ is not related to the trivial solution $f_{\lambda,0}\equiv 0$, but is just the limit of $\theta_{\lambda,\alpha}$ as $\alpha\searrow 0$.}
\begin{equation}\label{eq:lim}
\theta_{\lambda,0}=\frac{\pi}{2\sqrt{\lambda}}\qquad\text{and}
\qquad\lim_{\alpha\nearrow \infty}\theta_{\lambda,\alpha}=\frac{1}{2\sqrt{2 \lambda}} \, B\left(\frac{3}{4},\frac{1}{2}\right).
\end{equation} 
\end{lemma}

\begin{proof}
The integral on the right-hand side of \eqref{eq:tet} exists because the singularity behaves like ${(1-\tau)^{-1/2}}$ as $\tau \to 1$. 
Therefrom, the regularity assertion is clear. Let us now show that $\theta_{\lambda,\alpha}$ is strictly decreasing with respect to $\alpha$.
To this aim we fix $\tau \in(0,1)$ and define the function 
\begin{align*}
g_{\tau}(\alpha) &:= (1-\beta)\tau^2 + \beta\\ 
&= (1-(1+\alpha^2)^{-1/2})\tau^2 + (1+\alpha^2)^{-1/2}, \qquad \alpha\geq0.
\end{align*}
Since 
\[
\theta_{\lambda,\alpha} = \sqrt{\frac{2}{\lambda}} \int_0^1\frac{1}{\sqrt{1-\tau^2}}\frac{g_{\tau}(\alpha)}{
\sqrt{1+g_{\tau}(\alpha)}}\, d\tau,
\]
we see that $\theta_{\lambda,\alpha}$ is strictly decreasing with respect to $\lambda$, and it suffices to show that the mapping
$\left[[0,\infty)\ni \alpha\mapsto g_{\tau}(\alpha)(1+g_{\tau}(\alpha))^{-1/2}\right]$
has a negative derivative for all $\tau\in(0,1)$, $\alpha > 0$. Indeed, since for such $\alpha$ and $\tau$ we have that
\[
\frac{\partial g_\tau}{\partial \beta} = 1 - \tau^2 > 0 \quad\text{ and }\quad \frac{\partial \beta}{\partial \alpha} = -\frac{\alpha}{(1+\alpha^2)^{3/2}} < 0, 
\]
it follows from the chain rule that
\[
\frac{d}{d\alpha}
\left(\frac{g_{\tau}(\alpha)}{
1+g_{\tau}(\alpha)}
\right)=\frac{g_\tau^\prime(\alpha)(2+g_\tau(\alpha))}{2(1+g_\tau(\alpha))^{3/2}}<0,\quad\text{for all $\tau\in(0,1)$.}
\]
In view of that $g_\tau(0)=1,$ the first equality in \eqref{eq:lim} follows.
Taking into consideration that $\lim_{\alpha\to \infty}g_\tau(\alpha) = \tau^2$, 
we obtain that
\begin{align*}
\lim_{\alpha\to\infty}\theta_{\lambda,\alpha}&=\sqrt{\frac{2}{\lambda}}\int_0^1 \frac{\tau^2}{\sqrt{1-\tau^4}}\, d\tau= \frac{1}{2\sqrt{2 \lambda}} \int_0^1s^{-1/4}(1-s)^{-1/2}\, ds\\[1ex]
&= \frac{1}{2\sqrt{2 \lambda}}\,  B\left(\frac{3}{4},\frac{1}{2}\right).
\end{align*}
This completes the proof.
\end{proof}

Recall that we  are interested in determining the solutions of \eqref{eq:m}
which are not only odd, but also of minimal period $2\pi$.
Thus, we are interested in determining the set of $\lambda$ and $\alpha$ such that
$\theta_{\lambda,\alpha} = \pi/2$. The following lemma provides an answer in terms of a function $\lambda \mapsto \alpha$. 

\begin{lemma}\label{L:24} Let $\lambda_*$ be the constant defined by the relation \eqref{eq:g*}. If $\lambda \not\in (\lambda_*,1]$ then $\theta_{\lambda, \alpha} \neq \pi/2$, but given $\lambda \in(\lambda_*,1]$ there exists a unique $\alpha(\lambda)\in[0,\infty)$
such that $\theta_{\lambda,\alpha(\lambda)}=\pi/2.$
The mapping
\[
(\lambda_*,1] \ni \lambda \mapsto \alpha(\lambda)\in[0,\infty)
\]
is smooth, bijective, and decreasing.
\end{lemma}
\begin{proof} In view of Lemma~\ref{L:mon}, we have that $\theta_{\lambda,\alpha}=\pi/2$ if and only if 
\[
\frac{1}{2\sqrt{2\lambda}} B\left(\textstyle{\frac{3}{4},\frac{1}{2}}\right)<\frac{\pi}{2}\leq \frac{\pi}{2\sqrt{\lambda}},
\]
which is equivalent to that $\lambda\in(\lambda_*,1]$. Since $[(\lambda,\alpha)\mapsto\theta_{\lambda,\alpha}]$ is smooth and $\partial_\alpha \theta_{\lambda,\alpha} < 0$ for $\alpha > 0$, we infer from the implicit function theorem that
$[\lambda \mapsto \alpha(\lambda)]$ is smooth as well. Then
\[
0 = \frac{d}{d\lambda} \theta_{\lambda,\alpha(\lambda)} = \partial_\lambda \theta_{\lambda,\alpha} + \left(\partial_\alpha \theta_{\lambda,\alpha}\right)  \alpha^\prime(\lambda),
\]
so that $\alpha^\prime(\lambda) < 0$ for $\lambda \in (\lambda_*,1)$ in view of Lemma~\ref{L:mon}. From \eqref{eq:lim} we infer that $\alpha(1)=0$ and $\lim_{\lambda \to \lambda_*} \alpha(\lambda) = \infty$.
\end{proof}

With these preparations done, we come to the proof of the main result as stated in Theorem \ref{T:main}:

\begin{proof}[Proof of Theorem \ref{T:main}] 
It follows from Proposition~\ref{P:eau} and Lemma~\ref{L:24} that the odd solutions of \eqref{eq:m} of minimal period $2\pi$ coincide with the set
\[
\{(\lambda, \pm f_\lambda)\,:\, \lambda \in (\lambda_*,1] \},
\]
where we simply write $f_\lambda := f_{\lambda,\alpha(\lambda)}$. In order for those solutions to be physically realistic, we still have to require that $f_\lambda\in\U$, i.e. that $\|f\|_{C(\s)} < h$. The maximum of $|f_\lambda|$ is achieved at $x=\theta_{\lambda,\alpha(\lambda)}=\pi/2$, and we infer from \eqref{eq:u2} that $\lambda \in (\lambda_*,1]$ must additionally satisfy
\begin{equation}\label{eq:cen}
f_\lambda(\pi/2)=\sqrt{2\lambda^{-1}\left(1-(1+\alpha^2(\lambda))^{-1/2}\right)}<h.
\end{equation}

Let us first assume that $h^2 < 2\lambda_{*}^{-1}$. Since, in view of Lemma~\ref{L:24},
\[
2\lambda^{-1} \left(1-(1+\alpha^2(\lambda))^{-1/2}\right)\; \nearrow_{\lambda\to\lambda_*}\; 2\lambda_{*}^{-1}
\]
we find a unique $\lambda_h>\lambda_*$ with the property that 
\[
\max \left|f_{\lambda_h}\right| = \sqrt{2\lambda_h^{-1} \left(1-(1+\alpha^2(\lambda_h))^{-1/2}\right)}=h.
\]
By recalling that $\lambda = g(\varrho_+ - \varrho_-) / \gamma$, we infer the main part of the theorem, including $(i)$.

If instead $h^2 = 2\lambda_{*}^{-1}$, then $f_\lambda \in \U$ for all $\lambda \in (\lambda_*,1]$ and Lemma~\ref{L:24} implies that
\[
|f_\lambda| \nearrow_{\lambda\to\lambda_*} \sqrt{2 / \lambda_*}, \quad\text{ while }\quad f_\lambda'(0)=\alpha(\lambda)\nearrow _{\lambda\to\lambda_*}\infty.
\]
We have thus shown that $(ii)$ is valid, and assertion $(iii)$ follows similarly.
\end{proof}

\begin{rem} Differentiating  relation \eqref{eq:u1} with respect to $\lambda$ and using a maximum principle argument,  shows that
$f_{\lambda_1}> f_{\lambda_2}$ on $(0,\pi)$ provided $\lambda_*<\lambda_1<\lambda_2\leq1.$
The evolution of the solution  $f_\lambda$ with respect to $\lambda\in(\lambda_*,1]$ is pictured in Figure \ref{Figure3}.
\end{rem}

\begin{figure}
\includegraphics[clip=true, angle=0,  trim = 0 -10 0 0, width=0.4\linewidth]{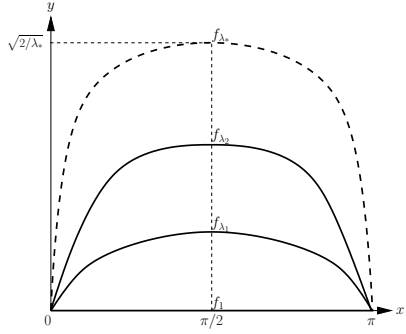}
\caption{\small Steady-states $f_\lambda$ on $\Sigma_1$, $\lambda_*<\lambda_2<\lambda_1<1.$}
\label{Figure3}
\end{figure}

\begin{proof}[Proof of Corollary~\ref{C:1}] 
The constant $\gamma_{l,h}$ is defined similarly to $\gamma_h$, and ensures that 
$l^{-1}f_\gamma(l\cdot)$ remain in $\U$ for all $\gamma\in[\ov\gamma_1,\gamma_{l,h})$ (see the proof of Theorem \ref{T:main}).
If $l$ is large enough relation \eqref{eq:u2} shows  that $\gamma_{l,h}=\gamma_*.$
 The relation \eqref{union} now follows in view of that the global bifurcation branch $\Sigma_l$ consists exactly of functions with minimal period $2\pi/l$.
\end{proof}

\section{Description of the bifurcation branches $\cC_l$}

Let us return to the setting of Theorem~\ref{T:P3}. Define
\[
\widetilde\Sigma_l:=\{(\gamma, f_\gamma(\cdot+\pi/2l))\,:\, (\gamma, f_\gamma)\in\Sigma_l\} \cup\{(\ov\gamma_l,0)\}, \qquad 1 \leq l \in \N.
\]
Since the functions $f_\gamma$ are odd (cf. \eqref{eq:fdef}), the smooth curve $\widetilde\Sigma_l$ consists of even functions. 
Hence, it must be a subset of the maximal connected component of $\cS$ to which $(\ov\gamma_l,0)$ belongs,  i.e. $\widetilde\Sigma_l\subset\cC_l$. 
We shall prove that the converse is also true. 
This means that the branches $\cC_l$ consist, with the exception of the trivial solution  $(\ov\gamma_l,0)$,  exactly of even functions with minimal period $2\pi/l$. 
Theorem~\ref{T:main} may be  then used to describe  the global bifurcation branches $\cC_l$. 

\begin{thm}\label{T:main2}
Given $l\in\N, l\geq 1,$ we have that $\cC_l=\widetilde\Sigma_l$.
\end{thm} 

Since the even solutions  of \eqref{eq:mean} near $(\ov\gamma_l,0)$ lay either on the trivial curve $\Sigma$
or on $\cC_l, $  we conclude that  $\widetilde\Sigma_l$ and $\cC_l$ coincide in a small neighbourhood of $(\ov\gamma_l,0)$. Hence, at least in small neighbourhood of $(\ov\gamma_l,0)$, the (seemingly) arbitrary constant in \eqref{eq:mean} is zero. Even more holds:

\begin{lemma}\label{L:3} Let $(\gamma,f)\in(0,\infty)\times\W$ be a solution of \eqref{eq:mean}.
We then have
\[
\frac{f''}{(1+f'^2)^{3/2}} + \lambda f=0\quad\text{in $\s$}.
\]
Moreover, if $2\pi/l, l\geq 1,$ is the minimal period of $f,$ then
$f(\cdot-\pi/2l)$ is  an odd solution of \eqref{eq:mean}.
\end{lemma}
\begin{proof}
Assume by contradiction that we would find a solution $(\gamma,f)\in(0,\infty)\times\W$ of \eqref{eq:mean} such that
\[
\frac{f''}{(1+f'^2)^{3/2}} + \lambda f = c\quad\text{in $\s$},
\]
with a constant $c\neq 0.$
Let $g:=f- c/\lambda.$
Then $g$ is an even solution of \eqref{eq:m}, but $g$ has no longer integral mean equal to $0$.
Since $g'(0)=0$, it must hold $g(0)\neq0.$
Otherwise, $g=0,$ meaning that $f= c/\lambda,$ which  contradicts $f\in\W$.

We may, without loss of generality, assume that $g(0)>0$.
There exists a positive time $T_c>0$ such that $g>0$ on $[0,T_c)$ and $g(T_c)=0.$
Indeed, if this is not the case,  we infer from \eqref{eq:m} that $g''<0.$
This is in contradiction with the fact that $g$ is periodic.
The function $g$ is a $4T_c-$periodic function on $\R,$ since
it must hold that
\begin{equation}\label{eq:even}
g(x)=\left\{
\begin{array}{rllll}
g(x),&& 0\leq x\leq T_c,\\[1ex]
-g(2T_c-x),&& T_c\leq x\leq 2T_c,\\[1ex]
-g(x-2T_c),&& 2T_c\leq x\leq 3T_c,\\[1ex]
g(4T_c-x),&& 3T_c\leq x\leq 4T_c.
\end{array}
\right.
\end{equation}
Moreover,  $g$ is $2\pi-$periodic, so that $T_c=\pi/(2k),$ for some $k\in\N.$ 
We conclude that $g$ has integral mean zero, 
which is in contradiction with  $f\in\W$ and $c\neq0.$
Thus, $f$ must solve \eqref{eq:m}.
The relation \eqref{eq:even}  then holds also for~$f$, provided that $T_c=\pi/2l$. This completes the proof.
\end{proof}
In virtue of Corollary \ref{C:1}, Proposition \ref{P:eau},  and the proof of Lemma \ref{L:3} we conclude:

\begin{rem}\label{R:2}The solutions of \eqref{eq:mean} are, up to translation, odd. 
Moreover, problem \eqref{eq:mean} has no solutions  $(\gamma,f)\in(0,\infty)\times \U$ if $\gamma>\gamma_h.$
\end{rem}

With this preparation done, the proof  of Theorem \ref{T:main2} is immediate.
\begin{proof}[Proof of Theorem \ref{T:main2}] Lemma \ref{L:3}
shows that the mapping
\[
\cup_{l=1}^\infty \left(\Sigma_l \cup\{(\ov\gamma_l,0)\}\right) \to \cS, \qquad \Sigma_l \cup\{(\ov\gamma_l,0)\}\ni(\gamma,f)\mapsto(\gamma, f(\cdot+\pi/2l)),
\]
is one-to-one and onto. 
Recall that  $\cS$ is the closure  of the set of nontrivial solutions of \eqref{eq:mean} in $(0,\infty)\times\W,$
and the union $\cup_{l=1}^\infty \left(\Sigma_l \cup\{(\ov\gamma_l,0)\}\right)$ is, in virtue of Corollary \ref{C:1}, disjoint.
The image of $\Sigma_l \cup\{(\ov\gamma_l,0)\}$ under this mapping is $\widetilde\Sigma_l,$ hence $\cC_l=\widetilde\Sigma_l$.
\end{proof}

\section*{Appendix}
It is clear form Proposition \ref{P:1} and  Corrolary \ref{C:1} that the $T-$periodic solutions of \eqref{eq:mean} are exactly the elements of  the set
\[
\mathcal{S}_T:=\left\{f\left(\frac{2\pi}{T}\cdot\right)\,:\, f\in\mathcal{S}_{2\pi}\right\}.
\]
The volume constrain in \eqref{eq:mean}  for $T-$periodic solutions reads as
$
\int_{0}^Tf\, dx=0.
$
In view of Remark  \ref{R:2} we show that each $T-$periodic  solution $f$
of \eqref{eq:mean} corresponds to a unique function $\theta$ describing the evolution of a mathematical pendulum (or the bending of an elastic rod):
\begin{equation}\label{eq:MP}
\theta''+\lambda\sin(\theta)=0,
\end{equation}
cf. also \cite{OS}.
We set $\lambda:=g/l$, with $l$ denoting the length of the pendulum. 
We refer to \cite{A, BT}
 for the deduction of \eqref{eq:MP}.
Herein, of   interest are only the solutions of \eqref{eq:MP} which satisfy 
\begin{equation}\label{eq:MPcond}
\text{ $|\theta|<\pi/2$ and  $\theta(0)=0$.}
\end{equation}
Particularly, solutions of \eqref{eq:MP}-\eqref{eq:MPcond} are odd. 
We now state:

\begin{thm} \label{T:correspondence}
There exists a one-to-one correspondence between the  even solutions $f$ of \eqref{eq:mean}
and the odd solutions $\theta$ of
\eqref{eq:MP}-\eqref{eq:MPcond}.

Given $s\in\R,$ $\theta(s)$ is the angle between the tangent to $\Gamma(f)$ at $z(s)$ and the $Ox$-axis, with $z:\R\to\R$
a parametrisation of  $\Gamma(f)$ by the arc length.  
\end{thm}
\begin{proof} Take first $f$ to be an even solution of \eqref{eq:mean}.
We define the function $p:\R\to\R$ by
\[
 p(x)=\int_{0}^x\sqrt{1+f'^2(t)}\, dt, \qquad x\in\R.
\]
This mapping   is bijective, and let $z$ denote its inverse.
Let $\theta:\R\to\R$ be given by 
\begin{equation}\label{eq:A4}
 \theta(s)=\arctan f'(z(s)),\qquad s\in\R.
\end{equation}
If $L=p(T),$ then $\theta$ is $L-$periodic.
Indeed, we have that
\[
L=p(z(s+L))-p(z(s))=\int_{z(s)}^{z(s+L)}\sqrt{1+f'^2(t)}\, dt=\int_{z(s)}^{z(s)+T}\sqrt{1+f'^2(t)}\, dt,
\]
hence $z(s+L)=z(s)+T.$
Being a composition of odd functions, $\theta$ is also odd.
Given $s\in\R$,
\[
\theta'(s)=\frac{f''(z(s))}{1+f'^2(z(s))}z'(s)=-\lambda f(z(s)) z'(s) \sqrt{1+f'^2(z(s))}=-\lambda f(z(s)),
\]
due to
\begin{equation}\label{eq:A3}
z'(s)=\frac{1}{p'(z(s))}=\frac{1}{(1+f'^2(z(s)))^{1/2}}.
\end{equation}

Hence
\[
\theta''(s)=\lambda f'(z(s))z'(s)=-\lambda \frac{f'(z(s))}{\sqrt{1+f'^2(z(s))}}=-\lambda \sin(\theta(s)),
\]
and $\theta$ is a $L-$periodic  solution of \eqref{eq:MP}-\eqref{eq:MPcond}.

Conversely, given $\theta$ solution of \eqref{eq:MP}-\eqref{eq:MPcond}, we let $z:\R\to\R$ be the function defined by
\[
z(s):=\int_0^s\cos(\theta(t))\, dt, \qquad s\in\R,
\]
and write $p$ for its  inverse. 
Setting $ T:=\int_0^L\cos(\theta(t))\, dt,$
we have that $p(x+T)=L+p(x)$ for all $x\in\R.$
It can be then easily verified that $f:\R\to\R,$
\[
 f(x):=-\frac{\theta'(0)}{\lambda}+\int_0^{x}\tan(\theta(p(t)))\, dt, \qquad x\in\R,
\]
is a $T-$periodic and even solution of \eqref{eq:mean}.
\end{proof}

With this observation, our result stated in Theorem \ref{T:main} rewrites for the mathematical pendulum equation as follows:
\begin{cor}\label{C:3}
There exists a smooth curve $\theta_\lambda,$ $\lambda\in(\lambda_*,1],$ consisting of $L_\lambda-$periodic solutions  of \eqref{eq:MP}-\eqref{eq:MPcond}
with the property that
\[
\sup |\theta|=\arctan(\alpha(\lambda)))\nearrow _{\lambda\to\lambda_*} \pi/2.
\]
\end{cor} 

\begin{rem}\label{R:dec}
It is worth noticing that the period $L_\lambda$ of these solutions is strictly decreasing with respect to $\lambda.$
Indeed, it holds that
\[
L_\lambda=\frac{2}{\sqrt{\lambda}}\int_{-\pi/2}^{\pi/2}\frac{d\theta}{\sqrt{1-\sin^2(\arctan(\alpha(\lambda))/2)\sin^2(\theta)}},
\]
and since $\alpha$ decreases with respect to $\lambda$ we obtain the desired conclusion.
 Though $L_\lambda$ can be calculated in terms of  elliptic integrals, it is in general difficult  to specify for which solution $\theta$ of \eqref{eq:MP}-\eqref{eq:MPcond} of period $L>0$ it holds that  $z(L)=2\pi,$
 that is the corresponding solution of \eqref{eq:mean} has period $2\pi.$
Furthermore, the result stated in Corollary \ref{C:3} can not be obtained via standard bifurcation theorems, since the period of $L_\lambda$ must 
decrease with respect to $\lambda.$ 
These facts serve as a motivation for our approach.
\end{rem}

\subsection*{Acknowledgement}
The authors are grateful to the anonymous referee for pointing out some ambiguity in the notation of the preliminary version of the paper.

\end{document}